\documentclass[12pt]{amsart}
\usepackage{latexsym}
\usepackage{amsmath}
\usepackage{epsfig}
\usepackage{epic}
\usepackage{amssymb}

\newtheorem{theorem}{Theorem}[section]
\newtheorem{lemma}[theorem]{Lemma}
\newtheorem{corollary}[theorem]{Corollary}

\newtheorem{conjecture}[theorem]{Conjecture}
\newtheorem{question}[theorem]{Question}

\theoremstyle{definition}

\theoremstyle{remark}

\numberwithin{equation}{section}

\newcommand{\ba}{\backslash}

\newcommand{\subproof}{\begin{proof}[Subproof]}

\begin{document}

\sloppy

\title{Connectivity Functions and Polymatroids}

\author{Susan Jowett}
\address{School of Mathematics Statistics and Operations Research,
Victoria University of Wellington}
\email{swoppit@gmail.com}
\thanks{Susan Jowett's research was supported by an MSc scholarship from Victoria University of Wellington.}

\author{Songbao Mo}
\address{School of Mathematics Statistics and Operations Research,
Victoria University of Wellington}
\email{songbao.mo@gmail.com}

\author{Geoff Whittle}
\address{School of Mathematics Statistics and Operations Research,
Victoria University of Wellington}
\email{geoff.whittle@vuw.ac.nz}
\thanks{Geoff Whittle's research was supported by a grant from the Marsden Fund of New Zealand}

\date{}

\begin{abstract}
A {\em connectivity function on} a set $E$ is a function $\lambda:2^E\rightarrow \mathbb R$
such that $\lambda(\emptyset)=0$, that $\lambda(X)=\lambda(E-X)$ for all $X\subseteq E$ and
that $\lambda(X\cap Y)+\lambda(X\cup Y)\leq \lambda(X)+\lambda(Y)$ for all 
$X,Y \subseteq E$. Graphs, matroids and, more generally, polymatroids have associated
connectivity functions. We introduce a notion of duality for polymatroids and prove that
every connectivity function is the connectivity function of a self-dual
polymatroid. We also prove that
every integral connectivity function is the connectivity function of a half-integral 
self-dual polymatroid.
\end{abstract}

\maketitle

\section{Introduction}

Let $E$ be a finite set, and  $\lambda$ be a function from the power set of $E$ into the
real numbers. Then $\lambda$ is {\em symmetric} if $\lambda(X)=\lambda(E-X)$ for all 
$X\subseteq E$; $\lambda$ is {\em submodular} if 
$\lambda(X\cap Y)+\lambda(X\cup Y)\leq \lambda(X)+\lambda(Y)$
for all $X,Y\subseteq E$; and $\lambda$ is {\em normalised} if  $\lambda(\emptyset)=0$.
If $\lambda$ is symmetric, submodular and normalised, then 
we say that $\lambda$ is a {\em connectivity function} with {\em ground set} $E$. We also 
say that $\lambda$ is a connectivity function {\em on} $E$. If $\lambda$ is a connectivity
function on $E$, then $\lambda$ is {\em integer-valued} if $\lambda(X)\in\mathbb Z$ 
for all $X\subseteq E$.
The  connectivity function $\lambda$ is {\em unitary} if 
$\lambda(\{x\})\leq 1$ for all $x\in E$.

Graphs and matroids have natural associated connectivity functions. These auxiliary structures 
capture vital information. It turns out that a number of quite fundamental properties of graphs and matroids hold
at the level of general connectivity functions. In particular this is the case
for properties associated with
branch-width and tangles of graphs and matroids. This is implicit --- but clear on a 
close reading --- in the paper
of Robertson and Seymour \cite{rsx}.  More explicit results  for connectivity functions are
proved in Geelen, Gerards and Whittle \cite{ggw1},
Clark and Whittle \cite{cw}, Hundertmark \cite{hu}, and  Grohe and Schweitzer \cite{gs}. 
  
Given that we can prove quite strong theorems for connectivity functions, the study
of these structures is well motivated and this paper forms part of that study.
The natural question arises as to just how general connectivity functions are.
The main purpose of this paper is to give an answer to that question. Polymatroids
are defined in the next section. We prove that
every connectivity function is the connectivity function of an associated polymatroid,
and every integral connectivity function is the connectivity function of an associated
half-integral polymatroid. The proofs of these facts are quite simple --- almost unnervingly
so --- but the results are apparently new and we believe that they are worth reporting.
Moreover, our main result surprised at least one of us as a number of naturally arising
connectivity functions seem to have little to do with polymatroids. 

As well as proving the above results we introduce a new notion of duality for 
polymatroids.  Via this duality we get stronger theorems. Every connectivity function is
the connectivity function of a {\em self-dual} polymatroid. An interesting feature of this
notion of duality is that, when restricted to the class of matroids, it gives a duality
that is subtly different from usual matroid duality. 

The results of this paper had their genesis in the MSc thesis of Mo \cite{songbao} and
were further developed in the MSc thesis of Jowett \cite{susan}. These theses also contain 
a number of other results on connectivity functions and their connection with polymatroids.

Since writing the first draft of this paper we have become aware of a paper of 
Mat\'u\v{s} \cite{matus}. While our perspective and terminology is quite different from those of 
Mat\'u\v{s} the fact is that a number of the results of this paper follow from results of
his. In particular our Lemmas~4.2 and 4.3 (ii) and (iii) follow from Theorem~1 of 
\cite{matus}. At a deeper level it is clear that most of the key ideas for which this paper
could claim originality are already present in \cite{matus}. 
The existence of Mat\'u\v{s}' paper came as a considerable surprise to as
as we believed throughout that we were exploring perfectly new territory. 
On the other hand Mat\'u\v{s}' perspective is quite different from ours --- he is motivated
by problems in information theory and our primary motivation comes from matroid theory.
Moreover the two papers have very different styles of exposition. The two papers should 
appeal to different audiences and we believe that there is a real advantage in having
both papers in print.

\section{Preliminaries}

Recall that a {\em polymatroid} $P=(r,E)$ is a finite set $E$ together with a 
function $r:2^E\rightarrow \mathbb R$ that is {\em normalised}, that is,
$r(\emptyset)=0$, {\em submodular}, that is, $r(X\cap Y)+r(X\cup Y)\leq r(X)+r(Y)$ for all
$X,Y\subseteq E$, and {\em increasing},
that is, $r(X)\leq r(Y)$ for all $X\subseteq Y\subseteq E$. The polymatroid $P$ is
{\em integer valued} if $r(X)\in \mathbb Z$ for all $Z\subseteq E$. It is {\em half-integral}
if $r(X)\in\{\frac{x}{2}:x\in \mathbb Z\}$ for all $X\subseteq E$. 
If $r(X)\leq k$ for all $X\subseteq E$, then $P$ is a $k$-{\em polymatroid}. We know of no
case where $k$-polymatroids are of interest except when $k$ is a positive integer.
We define the {\em connectivity function} $\lambda_P$ of the polymatroid
$P$ by
$\lambda_P(X)=r_P(X)+r_P(E-X)-r_P(E)$ for all $X\subseteq E$. 
It is well known and easily verified that, if $P$ is a polymatroid, then $\lambda_P$ is indeed
a connectivity function.

Two special cases of polymatroids are of particular interest.
Observe that a 
matroid $M$, when defined via its rank function is 
just an integer-valued polymatroid with the additional property 
that $r(\{e\})\leq 1$ for all
$e\in E$. In other words, a matroid is an integer-valued 1-polymatroid.
Via this specialisation, the connectivity function $\lambda_M$ of the matroid
$M$ as defined in, for example Oxley \cite{ox92}, is  nothing more than
the connectivity function we obtain when we regard $M$ as a polymatroid. Specifically,
if $M$ is a matroid on $E$, then 
$\lambda_M(X)=r_M(X)+r_M(E-X)-r_M(E)$ for all $X\subseteq E$. Evidently 
connectivity functions of matroids are integer-valued and unitary.

Let $G=(V,E)$ be a graph. 
Then the  {\em connectivity function of $G$}, denoted
$\lambda_G$ is defined by $\lambda_G(X)=|V(X)|+|V(E-X)|-|V|$ for all $X\subseteq E$.
Connectivity functions of graphs capture vertex connectivity. For each vertex cut
of order $k$ in $G$ there is an associated partition $(X,E-X)$ of the edges such that
$\lambda_G(X)=k$; in fact there may be more than one, so that the connectivity function of a
graph gives more information than the vertex cuts of that graph. 
Of course connectivity functions of graphs are
integer-valued but they are not usually unitary as, if $e$ is an edge of $G$, then
$\lambda_G(\{e\})=2$ unless $e$ is a loop or is incident with a leaf.
Associated with a graph is its cycle matroid $M(G)$. The connectivity function of
$G$ is quite distinct from that of $M(G)$. Nonetheless, the two are related. For example,
it is proved in \cite{hicks} that, apart from essentially trivial exceptions, the 
branch-width of a graph and its cycle matroid are the same. In the language of 
connectivity functions this means that, apart from the 
same exceptions, the branch-width of $\lambda(G)$ is one greater than the branch width of
$\lambda_{M(G)}$. In this paper when we refer to the connectivity function of a graph
$G$ we will always mean $\lambda_G$ as defined above.

Associated with a graph $G=(V,E)$ we define an integer-valued set function $r_G$ on
$E$ by setting $r_G(X)=|V(X)|$ for all $X\subseteq E$. We say that $r_G$ is the {\em rank 
function} of $G$. Evidently $r_G$ knows nothing
about isolated vertices or vertex labels in $G$. Apart from that, the graph $G$ is determined
by its rank function. Another feature of $r_G$ is that it is the rank function of an
integer-valued  
2-polymatroid. The fact that graphs essentially correspond to a family of 
integer-valued 2-polymatroids is of some interest and we take the opportunity here of 
expanding a little on these relationships.

Let $M$ be a matroid on a set, say $V$, and let $E$ be a collection of subsets of $V$.
Define the function $r_P$ on $E$ as follows: 
$r_P(X)=r_M\left(\cup_{x\in X}\right)$ for all 
$X\subseteq E$. It is well known, see for example \cite[Theorem~11.1.9]{ox92}, that 
every integer-valued polymatroid can be obtained in this way. 
Let $M$ be a {\em free matroid} on $V$, that is $r(V)=|V|$. In essence, free matroids are
trivial matroids.
Let $E$ be a collection of subsets of $V$ of size at most 2. The polymatroids we construct
from free matroids via this construction are precisely the 2-polymatroids we constructed
from the edge sets of graphs in the previous paragraph.  
Note that this way of viewing graphs is nothing
more than the time-honoured way of viewing a graph as a collection of lines generated by 
pairs of points of a simplex. 

\section{Polymatroid Duality} 

Let $P=(r,E)$ be a polymatroid. For a set $X\subseteq E$, we let
\[
||X||_r=\sum_{x\in X}(r(\{x\})).
\]  
We define the set function $r^*$ on $E$ by setting 
\[
r^*(X)= r(E-X)+||X||_r-r(E)
\]
for all $X\subseteq E$. We call the pair $P^*=(r^*,E)$, the {\em dual} of $P$.
We will prove that $P^*$ is a polymatroid, but first note an elementary lemma.

\begin{lemma}
\label{bar-bar}
Let $P=(r,E)$ be a polymatroid and let $X$ and $Y$ be subsets with $X\subseteq Y\subseteq E$.
Then $r(Y)-r(X)\leq ||Y-X||_r$.
\end{lemma}

\begin{proof}
By submodularity, $r(Y-X)+r(X)\geq r(Y)$, so that
$r(Y)-r(X)\leq r(Y-X)$. Again, by submodularity, $r(Y-X)\leq \sum_{y\in Y-X}r(\{y\})$, so 
that $r(Y-X)\leq ||Y-X||_r$. Thus $r(Y)-r(X)\leq ||Y-X||_r$ as required.
\end{proof}

\begin{lemma}
\label{dual}
Let $P$ be be a polymatroid on $E$. Then the dual $P^*=(r^*,E)$ is a polymatroid on 
$E$.
\end{lemma}

\begin{proof} 
We need to show that $r^*$ is normalised, increasing and submodular. We have
\[
r^*(\emptyset)=r(E)+||\emptyset||_r-r(E)=0,
\]
so that $r^*$ is normalised. Assume that $X\subseteq Y\subseteq E$. Then
\begin{align*}
r^*(Y)-r^*(X)&=r(E-Y)+||Y||_r-r(E)-r(E-X)-||X||_r+r(E)\\
&=r(E-Y)-r(E-X)+(||Y||_r-||X||_r)\\
&=||Y-X||_r-(r(E-X)-r(E-Y)).
\end{align*}
However $Y-X=(E-X)-(E-Y)$ so that it follows from  
Lemma~\ref{bar-bar} that $r^*(Y)-r^*(x)\geq 0$,
that is, $r^*$ is increasing.

Now say that $X,Y\subseteq E$. Evidently $||X||_r+||Y||_r=||X\cup Y||_r+||X\cap Y||_r$.
Using this fact and submodularity we see that
\begin{align*}
& r^*(X)+r^*(Y)\\
=& r(E-X)+||X||_r-r(E)+r(E-Y)+||Y||_r-r(E)\\
\geq & r(E-(X\cup Y))+r(E-(X\cap Y))+||X\cup Y||_r+||X\cap Y||_r-2r(E)\\
=& r^*(X\cup Y)+r^*(X\cap Y).
\end{align*}
Thus $r^*$ is submodular and the lemma follows.
\end{proof}

An alternative notion of duality for integer-valued polymatroids
was introduced in \cite{wh}. For a fixed positive integer $k$,  the set function
$r^{*k}$ is defined, for all $X\subseteq E$, by
\[
r^{*k}(X)=r(E-X)+k|X|-r(E).
\]
In the case $k=1$, we have the usual dual for matroids. 
The $k$-dual of an integer-valued $k$-polymatroid $P$ is an integer-valued $k$-polymatroid,
which we denote by $P^{*k}$. Moreover $k$-duality enjoys two natural properties. First,
$k$-duality is an {\em involution} on the class of $k$-polymatroids, that is,
for any $k$-polymatroid $P$, we have $(P^{*k})^{*k}=P$. Second; $k$-duality interchanges
deletion and contraction, that is, for any $X\subseteq E$, we have 
$(P\ba X)^{*k}=P^{*k}/X$. Indeed, it is proved in \cite{wh} that $k$-duality is the only 
function on the class of $k$-polymatroids that enjoys both of these properties. The definition
of duality for polymatroids we have given here
is not restricted to $k$-polymatroids for any fixed $k$ and
is certainly different from $k$-duality, so 
something has to give. It turns out that our notion of duality is not 
in general an involution.
Despite this, we shall see that the situation is not so dire. Indeed, it has its appeal.

Let $P=(r,E)$ be a polymatroid. Recall that we denoted the connectivity function of 
$P$ by $\lambda_P$. An element $e\in E$ is {\em compact} if $r(\{e\})=\lambda_P(\{e\})$. 
We say that the polymatroid $P$ is {\em compact} if every element of $P$ is compact.
Intuitively compact elements are ones that do not ``stick out'' from the rest of the 
polymatroid. More formally, we have

\begin{lemma}
\label{spanned}
Let $P=(r,E)$ be a polymatroid. The element $e\in E$ is compact if and only if 
$r(E-\{e\})=r(E)$.
\end{lemma}

\begin{proof}
The element $e$ is compact if and only if 
$r(\{e\})=\lambda_P(\{e\})=r(E-\{e\})+r(\{e\})-r(E)$. This holds if and only if
$(E-\{e\})=r(E)$.
\end{proof}

A matroid is compact if and only if it has no coloops. The 2-polymatroid that we associate 
with a connected graph is compact if and only if the graph has no leaves. 
Given a polymatroid $P$, 
there is a natural compact polymatroid that we can associate with $P$ that has the same 
connectivity function as $P$. We consider this now.

Let $P=(E,r)$ be a polymatroid. Define the function $r^\flat\rightarrow \mathbb R$ by
\[
r^\flat(X)=r(X)+\sum_{x\in X}(\lambda(\{x\})-r(\{x\}))
\]
for all $X\subseteq E$. The pair $P^\flat=(E,r^\flat)$ is the {\em compactification}
of $P$.

It turns out that $P^\flat$ is a polymatroid and $\lambda_{P^\flat}=\lambda_P$. These
facts will follow from the connection with duality.

\begin{lemma}
\label{dual-compact}
Let $P=(E,r)$ be a polymatroid. Then the following hold.
\begin{itemize}
\item[(i)] $\lambda_{P^*}=\lambda_P$.
\item[(ii)] $P^*$ is compact.
\item[(iii)] $(P^*)^*=P^\flat$.
\end{itemize}
\end{lemma}

\begin{proof}
Consider (i). For a set $X\subseteq E$, we have $||X||_r+||E-X||_r=||E||_r$.
Using this fact and definitions we see that
\begin{align*}
& \lambda_{P^*}(X)\\
= &r^*(X)+r^*(E-X)-r^*(E)\\
= &r(E-X)+||X||_r-r(E)+r(X)+||E-X||_r-r(E)\\
&\text{~~~~~~~~}-r(\emptyset)-||E||_r+r(E)\\
= &\lambda_P(X).
\end{align*}

Consider (ii). Say $e\in E$. Then 
\begin{align*}
r^*(E-\{e\})&=r(\{e\})+||E-\{e\}||_r-r(E)\\
&=||E||_r-r(E)\\
&=r(\emptyset)+||E||_r-r(E)\\
&=r^*(E).
\end{align*}
Therefore $e$ is compact in $P^*$, and (ii) follows.

Consider (iii). Say $X\subseteq E$. Then
\begin{align*}
(r^*)^*(X)&=r^*(E-X)+||X||_{r^*}-r^*(E)\\
&=r(E-(E-X))+||E-X||_r-r(E)+||X||_{r^*}\\
&{\text ~~~~~~~~~} -r(E-E)-||E||_r+r(E)\\
&=r(X)+\sum_{x\in X}(\lambda_P(\{x\})-r(\{x\}))\\
&=r^\flat(X).
\end{align*}
Therefore $(P^*)^*=P^\flat$ as required.
\end{proof}

As an immediate consequence of Lemmas~\ref{dual} and \ref{dual-compact}, we obtain

\begin{corollary}
\label{compact-ok}
Let $P=(E,r)$ be a polymatroid. Then $P^\flat$ is a polymatroid and 
$\lambda_{P^\flat}=\lambda_P$.
\end{corollary}

Thus, while polymatroid duality is not an involution in general, it is an involution on the 
class of compact polymatroids. The situation is analogous to that of planar drawings
of graphs. The planar dual is always connected so that planar duality is not an 
involution, but planar duality is an
involution on the class of planar drawings of connected graphs.

Regarded as a polymatroid, a matroid is compact if and only if it has no coloops.
Say $e$ is a loop of the matroid $M$. With the usual notion of matroid duality 
$e$ becomes a coloop in the dual of $M$.
With the notion of duality given here, $e$ remains a loop in the dual. 
Apart from that, the two notions of duality coincide for matroids.

We now consider the connection with minors. Let $P=(r,E)$ be a polymatroid, and 
$A\subseteq E$. The {\em deletion} of $A$ from $P$, denoted $P\ba A$ is defined, for all
$X\subseteq E-A$, by $r_{P\ba A}(X)=r_P(X)$. The {\em contraction} of $A$ from $P$,
denoted $P/A$, is defined for all $X\subseteq E-A$, by $r_{P/A}(X)=r_P(X\cup A)-r_P(A)$.
These notions generalise familiar ones from matroid theory. 

We would like to say that, just as with matroids, deletion and contraction are interchanged
under duality, but this cannot be, since compactness can be lost 
by deletion. However compactness cannot be lost by contraction.

\begin{lemma}
\label{compact-contract}
Let $P=(E,r)$ be a compact polymatroid. Then $P/A$ is compact for any $A\subseteq E$.
\end{lemma}

\begin{proof}
Say $e\in E-\{a\}$. By Lemma~\ref{spanned}, $r(E-\{e\})=r(E)$. We then have
\begin{align*}
& \lambda_{P/A}(\{e\})\\
=& r_{P/A}(\{e\})+r_{P/A}((E-A)-\{e\})-r_{P/A}(E-A)\\
=& r_P(A\cup\{e\})-r_P(A)+r_P(E-\{e\})-r_P(A)-(r_P(E)-r_P(A))\\
=&r_P(A\cup\{e\})-r_P(A)\\
=&r_{P/A}(\{e\}),
\end{align*}
so that $P/A$ is indeed compact.
\end{proof}

Again the situation is analogous to that of plane graphs where connectivity can be
lost by deletion, but not by contraction. Combined with compactification, we do
obtain a nice relation under duality.

\begin{lemma}
\label{nice-dual}
Let $P=(E,r)$ be a polymatroid and $A\subseteq E$. Then
$(P/A)^*=(P^*\ba A)^\flat$.
\end{lemma}

\begin{proof}
First note that both $(P/A)^*$ and $(P^*\setminus A)^\flat$ are defined on the same set, that is $E-A$.
Consider $r_{(P^*\setminus A)^\flat}(X)$ for $X\subseteq E-A$. We have the following chain of equalities:
\begin{align*} 
r_{(P^*\setminus A)^\flat}(X)&=(r_{P^*\setminus A})(X)+\sum\limits_{a\in X}[(\lambda_{P^*\setminus A})(\{a\})-(r_{P^*\setminus A})(\{a\})]\\
&=(r_{P^*\setminus A})(X)+\sum\limits_{a\in X}[(r_{P^*\setminus A})((E-A)-\{a\})-(r_{P^*\setminus A})(E-A)]\\
&=r_{P^*}(X)+\sum\limits_{a\in X}[r_{P^*}((E-A)-\{a\})-r_{P^*}(E-A)]\\
&=r_P(E-X)+||X||_{r_P} -r_P(E) +\sum\limits_{a\in X}[r_P(A\cup\{a\})\\
&\hspace*{20pt}+||E-(A\cup\{a\})||_{r_P}-r_P(E)-(r_P(A)+||E-A||_{r_P}\\&\hspace*{20pt}-r_P(E))]\\
&=r_P(E-X)+||X||_{r_P} -r_P(E) +\sum\limits_{a\in X}[r_P(A\cup\{a\})+||E||_{r_P}\\
&\hspace*{20pt}-||A||_{r_P}-r_P(\{a\})-r_P(E)-r_P(A)-||E||_{r_P}+||A||_{r_P}\\&\hspace*{20pt}+r_P(E)]\\
&=r_P(E-X)+||X||_{r_P}-r_P(E)+\sum\limits_{a\in X}[r_P(A\cup \{a\})-r_P(\{a\})\\&\hspace*{20pt}-r_P(A)]\\
&=r_P(E-X)-r_P(E)+\sum\limits_{a\in X}[r_P(A\cup \{a\})-r_P(A)].
\end{align*}
Now consider $r_{(P/A)^*}(X)$ and consider the following chain of equalities:
\begin{align*}
r_{(P/A)^*}(X)&=(r_{P/A})((E-A)-X)+||X||_{r_{P/A}}-(r_{P/A})(E-A)\\
&=r_P(E-X)-r_P(A)+||X||_{r_{P/A}}-r_P(E)+r_P(A)\\
&=r_P(E-X)-r_P(E)+\sum\limits_{a\in X}(r_{P/A})(\{a\})\\
&=r_P(E-X)-r_P(E)+\sum\limits_{a\in X}[r_P(A\cup \{a\})-r_P(A)].
\end{align*}
Therefore $(P/A)^*=(P^*\setminus A)^\flat$
\end{proof}

Given the above correspondences, it seems natural to operate within the universe of compact
polymatroids. In this universe one could incorporate compactification in the definition of
deletion. Given that a polymatroid and its compactification have the same connectivity 
function, no real loss is incurred by taking this approach if our interest is in polymatroid
connectivity.

\section{Connectivity functions and polymatroids}

We say that a connectivity function $\lambda$ on $E$ is {\em matroidal} if there exists a 
matroid $M$ such that $\lambda=\lambda_M$. We say that $\lambda$ is {\em connected} if
$\lambda(X)>0$ whenever $X$ is a proper nonempty subset of $E$. Assume that $\lambda$ is a
connected matroidal connectivity function, say $\lambda=\lambda_M$. Another matroid
with the same connectivity function is $M^*$. It follows from work of Seymour \cite{se}
and Lemos \cite{le} that, if $r(M)\neq r(M^*)$, then these are the only matroids whose
connectivity functions are equal to $\lambda$. When $r(M)=r(M^*)$, there are cases where 
other matroids can have the same connectivity function. The situation is certainly highly
structured, but, even in the case where the only matroids with connectivity function $\lambda$
are $M$ and its dual, it is by no means straightforward 
to find the rank function of $M$ or $M^*$ from
$\lambda$. 

Let $M$ and $N$ be matroids on $E$. Define the function $r$ by $r(X)=r_M(X)+r_N(X)$ for all
subsets of $E$. It is well known and easily seen that $r$ is the rank function of a 
$2$-polymatroid on $E$. In particular, this holds when $N=M^*$. To eliminate
ambiguity caused by the distinction between the two types of duality for matroids, assume
that $M$ is a loopless matroid.
Then it is elementary to check that for all $X\subseteq E$,
we have 
\[
r_M(X)+r_{M^*}(X)=\lambda(X)+|X|.
\]
Thus we can canonically construct a 2-polymatroid $P$ from $\lambda$. Moreover, it is easily
checked that $\lambda_P=2\lambda$, so that, up to a scaling factor, we have constructed a 
2-polymatroid whose connectivity function is equal to $\lambda$. Alternatively, we could 
observe that the fractional polymatroid $P/2$ has connectivity function equal to 
$\lambda$. In themselves, these observations are not particularly interesting, except for the
fact that they generalise as we now show. 

Let $\lambda$ be a connectivity function on $E$. For $X\subseteq E$, we define $||X||_\lambda$
by 
\[
||X||_\lambda=\sum_{x\in X}\lambda(\{x\}).
\]

\begin{lemma}
\label{basic-stuff}
Let $\lambda$ be a connectivity function on $E$ and let $A$ and $B$ be subsets of $E$
with $A\subseteq B$. Then
\[
\lambda(B)-\lambda(A)\leq ||B-A||_\lambda.
\]
\end{lemma}

\begin{proof}
It follows from submodularity that, if $Z\subseteq E$, then $\lambda(Z)\leq ||Z||_\lambda$.
It also follows from submodularity that $\lambda(B)-\lambda(A)\leq \lambda(B-A)$.
Combining these two observations gives the result.
\end{proof}

\begin{lemma}
\label{get-poly}
Let $\lambda$ be a connectivity function on $E$ and define the set function $r$ on $E$ by
\[
r(X)=\lambda(X)+||X||_\lambda
\]
for all $X\subseteq E$. Then $r$ is the rank function of a polymatroid on $E$.
\end{lemma}

\begin{proof}
Evidently $r(\emptyset)=0$. Assume that $X\subseteq Y\subseteq E$. Then we have
\begin{align*}
r(Y)-r(X)&=\lambda(Y)+||Y||_\lambda -\lambda(X)-||X||_\lambda\\
&=||Y-X||_\lambda +\lambda(Y)-\lambda(X)\\
&=||(E-X)-(E-Y)||_\lambda -(\lambda(E-X)-\lambda(E-Y)).
\end{align*}
It follows from Lemma~\ref{basic-stuff} that 
$||(E-X)-(E-Y)||_\lambda -(\lambda(E-X)-\lambda(E-Y))\geq 0$. Hence 
$r$ is increasing.

Assume that $X,Y\subseteq E$. Observe that
$||X||_\lambda+||Y||_\lambda=||X\cup Y||_\lambda+||X\cap Y||_\lambda$. Using this
fact and the submodularity of $\lambda$, we have
\begin{align*}
r(X)+r(Y)&=\lambda(X)+\lambda(Y)+||X||_\lambda+||Y||_\lambda\\
&\geq \lambda(X\cup Y)+\lambda(X\cap Y)+||X\cup Y||_\lambda+||X\cap Y||_\lambda\\
&=r(X\cup Y)+r(X\cap Y).
\end{align*}
Hence $r$ is submodular and $r$ is the rank function of a polymatroid as claimed.
\end{proof}

We say that the polymatroid constructed via Lemma~\ref{get-poly} is the polymatroid 
{\em induced} by $\lambda$.

\begin{lemma}
\label{get-poly2}
Let $\lambda$ be a connectivity function on $E$ and let $P=(E,r)$ denote the
polymatroid induced by $\lambda$. Then the following hold.
\begin{itemize}
\item[(i)] $\lambda_P(X)=2\lambda(X)$ for all $X\subseteq E$. 
\item[(ii)] $P$ is compact.
\item[(iii)] $P$ is self dual. 
\end{itemize}
\end{lemma}

\begin{proof}
Observe that $r(E)=\lambda(E)+||E||_\lambda=||E||_\lambda$. We use this fact several times.
Say $X\subseteq E$. Then 
\begin{align*}
\lambda_P(X)&=r(X)+r(E-X)-r(E)\\
&=\lambda(X)+||X||_\lambda+\lambda(E-X)+||E-X||_\lambda-||E||_\lambda\\
&=\lambda(X)+\lambda(E-X)\\
&=2\lambda(X).
\end{align*}
Hence (i) holds. Consider (ii). Say $e\in E$.
\begin{align*}
r(E)-r(E-\{e\})&=||E||_\lambda-\lambda(E-\{e\})-||E-\{e\}||_\lambda\\
&=||\{e\}||_\lambda-\lambda(\{e\})\\
&=0.
\end{align*}
Hence $P$ is compact, so that (ii) holds. Consider (iii). Let $r^*$ denote the 
rank function of $P^*$. Then, for $X\subseteq E$,
\[
r^*(X)=r(E-X)+||X||_r-r(E).
\]
By definition, $||X||_r=2||X||_\lambda$. Also $r(E-X)=\lambda(E-X)+||E-X||_\lambda$,
and $\lambda(E-X)=\lambda(X)$. Hence we have 
\begin{align*}
r^*(X)&=\lambda(X)+||E-X||_\lambda+2||X||_\lambda-||E||_\lambda\\
&=\lambda(X)+||X||_\lambda\\
&=r(X).
\end{align*}
Hence $P$ is self dual.
\end{proof}

From the above lemmas we obtain

\begin{theorem}
\label{poly-cor}
Every connectivity function is the connectivity function of a compact, self-dual polymatroid.
\end{theorem}

\begin{proof}
Let $\lambda$ be a connectivity function on $E$ and let $P$ be the polynomial induced by
$\lambda$. Define $P/2$ by $P/2(X)=P(X)/2$ for all $X\subseteq E$. 
It is easily checked that $P/2$ satisfies the conditions of the
theorem.
\end{proof}

Specialising to the integer-valued case we obtain 

\begin{corollary}
\label{poly-cor2}
Every integer-valued connectivity function is the connectivity function of a half-integral
self-dual polymatroid.
\end{corollary}

The case of unitary integer-valued connectivity functions is of particular interest. Up to 
a scaling factor, these are captured by self-dual integral 2-polymatroids. This does suggest
that such 2-polymatroids are worth studying in their own right.

\end{document}